\documentclass[12pt,a4paper]{article}
\usepackage{amsmath}
\usepackage{amsfonts}
\usepackage{amsthm}
\usepackage{amssymb}
\usepackage{booktabs,amsmath}

\def\XXint#1#2#3{{\setbox0=\hbox{$#1{#2#3}{\int}$}
     \vcenter{\hbox{$#2#3$}}\kern-.5\wd0}}

\usepackage[english]{babel}
\usepackage{xcolor}
\usepackage{enumerate}

\theoremstyle{plain}
\newtheorem{lemma}{Lemma}[section]
 \newtheorem{prop}[lemma]{Proposition}
        
        \newtheorem{definition}{Definition}[section]

\newtheorem{lem}[lemma]{Lemma}
\newtheorem{theo}{Theorem}[section]

\newcommand{\R}{\mathbb{R}}

\newcommand{\M}{\mathbb{M}}
\newcommand{\K}{\mathbb{K}}
\newcommand{\Q}{\mathbb{Q}}
\newcommand{\divrg}{\textrm{div}\,}

\theoremstyle{definition}
\usepackage{tabularx}

\usepackage{amsmath,amsfonts,amssymb,amscd,lscape,amstext,amsthm}

\numberwithin{equation}{section}

{\setlength{\parindent}{0pt}
\title{ {\bf Stable determination of the Winkler subgrade coefficient in a nanoplate }}

\oddsidemargin 0.0in \evensidemargin 0.0in

\author{G. Alessandrini\thanks{
Universit\`{a} degli Studi di Trieste, Italy,  
(alessang@units.it), ORCID: 0000-0002-1975-1299 }  \and A. Morassi\thanks{ Universit\`{a}
degli Studi di Udine, Italy (antonino.morassi@uniud.it), ORCID: 0000-0002-2186-7688}\and E.
Rosset\thanks{ Universit\`{a} degli Studi di Trieste, Italy
(rossedi@units.it), ORCID:0000-0003-2354-0459}  \and E. Sincich\thanks{ Universit\`{a} degli Studi di Trieste, Italy
(esincich@units.it), ORCID: 0000-0002-7405-1842}\and S. Vessella\thanks{ Universit\`{a}
degli Studi di Firenze, Italy (sergio.vessella@unifi.it), ORCID:  0000-0003-0139-8690}}

\linespread{1.6}

\date{}

\setlength{\parindent}{0pt}

\begin{document}

\maketitle

\begin{center}
\noindent \textbf{Abstract} 
\end{center}
\noindent{
We study the inverse problem of determining the Winkler coefficient in a nanoplate resting on an elastic foundation and clamped at the boundary. The nanoplate is described within a simplified strain gradient elasticity theory for isotropic materials, under the Kirchhoff-Love kinematic assumptions in infinitesimal deformation. We prove a global H\"{o}lder stability estimate of the subgrade coefficient by performing a single interior measurement of the transverse deflection of the nanoplate induced by a load concentrated at one point.}

\medskip

\medskip
 
\noindent \textbf{Mathematical Subject Classifications (2010)}: Primary: 35R30; Secondary: 35J25, 86A22.

\medskip

\medskip

\noindent \textbf{Key words}:  Elastic nanoplate, Winkler coefficient, H\"{o}lder stability

\section{Introduction}
\label{Introduction}

The remarkable development of nanotechnology over the last two decades has favoured the widespread use of nanoplate-like structures in the realisation of micro- and nano-electromechanical systems. The applications are numerous, ranging from actuators to resonators for measuring small masses, to shape memory devices, and involve problems in mechanics, electronics, and medicine \cite{L-D-L-X-Z-H-L-2012}, \cite{S-T-L-T-2012}, \cite{H-K-M-2019}, \cite{XQRZL-2021}.

In many applications, nanoplates are constrained at their boundary and, in addition, interact with the supporting structure in interior regions. This interaction can be realised, for example, by adhesive material interposed between the nanoplate and support. Correct modelling of the interaction is obviously important for the effective design and optimisation of these devices \cite{F-R-Z-X-X-2019}.

Among a variety of approaches to model nanoplate interaction with the supporting structure, Winkler-like elastic medium is one of the most popular \cite{S-N-S-M-2017}, \cite{S-A-R-2018}. In this model, the interaction is described by a bed of linear elastic springs of stiffness $\kappa$, $\kappa \geq 0$, that are assumed to act in the direction orthogonal to the nanoplate. The springs are independent of each other, that is the deflection of each spring is not influenced by the other adjacent springs.

The accuracy of the Winkler foundation model obviously depends on the values assigned to the subgrade coefficient $\kappa$. In fact, resonance frequencies of the free vibrations change appreciably as the coefficient $\kappa$ varies. In addition, a significant spatial variability of $\kappa$ can induce non-standard trends in the eigenfunctions of the nanoplate compared, for example, to the case of absence of interaction, and this can be reflected in important changes in the overall structural behavior \cite{B-M-M-F-2022}. The problem is further complicated by the fact that reference values for $\kappa$ are difficult to establish, also due to the variability of the adhesive interface and of the nanoplate and support materials. The situation is quite different from the soil-foundation interaction for real-scale buildings and constructions, for which approximate values of $\kappa$ in relation to the type of soil have long been available in the literature  \cite{Win}, \cite{CG}. For the reasons stated above, the inverse problem of determining the subgrade coefficient for a nanoplate is quite challenging and does not seem to have been systematically studied yet.

In this paper we consider the stability issue for the inverse problem of determining the subgrade coefficient of a nanoplate resting on a Winkler-like elastic foundation {}from the measurement of the transverse deflection induced at interior points of the nanoplate by a prescribed static load condition. The nanoplate is modelled within a simplified strain gradient elasticity theory developed by Lam et al. \cite{L-Y-C-W-T-2003}, under the Kirchhoff-Love kinematic assumptions in infinitesimal transverse deformation \cite{W-Z-Z-C-2011}, \cite{M-M-2013}. The nanoplate is described as a thin cylinder $\Omega \times (-t/2,t/2)$, where the bounded Lipschitz domain $\Omega$ in $\R^2$ is the middle surface and $t$ is the uniform thickness, $t << diam (\Omega)$. The nanoplate is assumed to be clamped at the boundary and the loads are represented by a single concentrated force acting transversally at an internal point $P_0$ of $\Omega$. Under the above assumptions, the transverse deflection $w: \overline{\Omega} \rightarrow \R$ of the nanoplate satisfies the sixth-order Dirichlet boundary value problem
\begin{center}
	{\small
		\( {\displaystyle 
			\begin{array}{lr}
				\divrg (\divrg ( (\mathbb P +\mathbb P ^h)\nabla^2 w)) 
				-\divrg (\divrg (\divrg (\mathbb Q \nabla^3 w)))
				+ \kappa w= f\delta_{P_0},& \mathrm{in}\ \Omega,
				\vspace{0.25em}\\
				w =0, & \mathrm{on}\ \partial \Omega,
				\vspace{0.25em}\\
				w,_n =0, & \mathrm{on}\ \partial \Omega,
				\vspace{0.25em}\\
				w,_{nn} =0, & \mathrm{on}\ \partial \Omega.
				\vspace{0.25em}\\
			\end{array}
		}
		\) \vskip -10.7em
		\begin{eqnarray}
			& & \label{eq-intro-3I-1}\\
			& & \label{eq-intro-3I-2}\\
			& & \label{eq-intro-3I-3}\\
			& & \label{eq-intro-3I-4}
		\end{eqnarray}
	} 
\end{center}
%

where $n$ is the unit outer normal to $\partial \Omega$. In equation \eqref{eq-intro-3I-1}, $\mathbb P$ is the fourth-order tensor describing the material response in classical Kirchhoff-Love theory, whereas $\mathbb P^h$ and $\mathbb Q$ are  fourth-order and sixth-order tensors respectively that incorporate small length scale effects. We refer to \cite{K-M-Z-2022} for more details on the derivation of the mechanical model. 

Given the concentrated force $f \delta_{P_0}$ and the Winkler coefficient $\kappa \in L^\infty (\Omega)$ and $\kappa\geq 0$ and  for $\mathbb P$, $\mathbb P^h$, $\mathbb Q$ satisfying suitable symmetry and strong convexity conditions, the direct problem \eqref{eq-intro-3I-1}--\eqref{eq-intro-3I-4} admits a unique solution $w \in H^3_0(\Omega)$; see Proposition \ref{prop-4-1} for a precise statement.

Our goal is the stability issue for the determination of the coefficient $\kappa$ {}from a single measurement of the deflection $w$ inside $\Omega$. In Theorem \ref{theo-7-1} we show that, for isotropic material and under suitable additional regularity of the elastic coefficients, if $w_i \in H^3_0(\Omega)$ is the solution to \eqref{eq-intro-3I-1}--\eqref{eq-intro-3I-4} for the Winkler coefficient $\kappa=\kappa_i \in L^\infty (\Omega) \cap H^s (\Omega)$, for some $0<s<1$, $i=1,2$, and if, for a given $\epsilon >0$,
\begin{equation}
	\label{eq:introd-w1-e-w2-vicini}
	\begin{aligned}{}
		& \| w_1 - w_2 \|_{L^2 (\Omega)} \leq \epsilon f,
	\end{aligned}
\end{equation}
then, for every $\sigma >0$, we have 
\begin{equation}
	\label{eq:stabilità}
	\begin{aligned}{}
		& \| \kappa_1 - \kappa_2 \|_{L^2 (\Omega_\sigma)} \leq C \epsilon^\beta,
	\end{aligned}
\end{equation}
where $\Omega_\sigma = \{ x \in \Omega | \ dist(x,\partial \Omega) > \sigma  \}$ and the constants $C>0$, $\beta \in (0,1)$ depend on the a priori data and on $\sigma$.

The method we adopted to prove \eqref{eq:stabilità} has similarities to that used in \cite{A-M-R-V-2015} to obtain a global H\"older estimate of the Winkler coefficient for the classical fourth-order Kirhhoff-Love's plate operator, and is inspired by a technique used by Alessandrini \cite{Al} for an inverse problem with interior measurements arising in hybrid imaging. The key ingredients of our proof of Theorem \ref{theo-7-1} are quantitative versions of the unique continuation principle for solutions to the equation $	\divrg (\divrg ( (\mathbb P +\mathbb P ^h)\nabla^2 w)) 
-\divrg (\divrg (\divrg (\mathbb Q \nabla^3 w)))
+ \kappa w=0$, which allow to keep under control the possible vanishing rate of $w$. These results are in the form of the Lipschitz propagation of smallness property (Proposition \ref{prop-8-1}) and the $A_p$ property (Proposition \ref{prop-9-1}), and are consequences of some basic unique continuation tools, such as the three spheres inequality and the doubling inequality which were recently derived in \cite{M-R-S-V-2024}  for the nanoplate operator for isotropic material. Another useful mathematical tool is a regularity result for a solution to the inhomogeneous equation $	\divrg (\divrg ( (\mathbb P +\mathbb P ^h)\nabla^2 w)) 
-\divrg (\divrg (\divrg (\mathbb Q \nabla^3 w)))
=g$ with the coefficients and the function $g$ belonging to suitable fractional Sobolev spaces (Proposition \ref{prop-5a-1}).

The plan of the paper is as follows. In Section \ref{Notation}  we collect some notation and definitions. Section \ref{Direct} contains the formulation of the direct problem and a point-wise lower bound of the solution to \eqref{eq-intro-3I-1}--\eqref{eq-intro-3I-4} in a neighbourhood of the point $P_0$ where the concentrated force is acting. Section \ref{Inverse} is devoted to the formulation and analysis of the inverse problem.

\section{Notation}
\label{Notation}

We shall denote by $B_r(P)$ the open disc in $\R^2$ of radius $r$ and
center $P$.

For any $U \subset \R^2$ and for any $r>0$, we denote
\begin{equation}
  \label{eq:2.int_env}
  U_{r}=\{x \in U \ |\  \textrm{dist}(x,\partial U)>r
  \}.
\end{equation}
\begin{definition}
  \label{def:2.1} (${C}^{k,\alpha}$ regularity)
Let $\Omega$ be a bounded domain in ${\R}^{2}$. Given $k,\alpha$,
with $k=0,1,2,...$, $0<\alpha\leq 1$, we say that a portion $S$ of
$\partial \Omega$ is of \textit{class ${C}^{k,\alpha}$ with
constants $\rho_{0}$, $M_{0}>0$}, if, for any $P \in S$, there
exists a rigid transformation of coordinates under which we have
$P=O$ and
\begin{equation*}
  \Omega \cap B_{\rho_{0}}(O)=\{x=(x_1,x_2) \in B_{\rho_{0}}(O)\quad | \quad
x_{2}>\psi(x_1)
  \},
\end{equation*}
where $\psi$ is a ${C}^{k,\alpha}$ function defined in
$I_{\rho_0}=(-\rho_{0},\rho_{0})$
satisfying
\begin{equation*}
\psi(0)=0,
\end{equation*}
\begin{equation*}
\psi' (0)=0, \quad \hbox {when } k \geq 1,
\end{equation*}
\begin{equation*}
\|\psi\|_{{C}^{k,\alpha}(I_{\rho_0})} \leq M_{0}\rho_{0}.
\end{equation*}

\medskip
\noindent When $k=0$, $\alpha=1$, we also say that $S$ is of
\textit{Lipschitz class with constants $\rho_{0}$, $M_{0}$}.
\end{definition}

  We use the convention to normalize all norms in such a way that their
  terms are dimensionally homogeneous with the argument of the norm and coincide with the
  standard definition when the dimensional parameter equals one.
  For instance, the norm appearing above is meant as follows
\begin{equation*}
  \|\psi\|_{{C}^{k,\alpha}(I_{\rho_0})} =
  \sum_{i=0}^k \rho_0^i
  \|\psi^{(i)}\|_{{L}^{\infty}(I_{\rho_0})}+
  \rho_0^{k+\alpha}|\psi^{(k)}|_{\alpha, I_{\rho_0}},
\end{equation*}
where
\begin{equation*}
|\psi^{(k)}|_{\alpha, I_{\rho_0}}= \sup_
{\overset{\scriptstyle x_1, \ y_1\in I_{\rho_0}}{\scriptstyle
x_1\neq y_1}} \frac{|\psi^{(k)}(x_1)-\psi^{(k)}(y_1)|}
{|x_1-y_1|^\alpha}
\end{equation*}
and $\psi^{(i)}$ denotes the $i$th-order
derivative of $\psi$.

Similarly, given a function $u:\Omega\mapsto \R$, where $\partial
\Omega$ satisfies Definition \ref{def:2.1}, and denoting by
$\nabla^i u$ the vector which components are the derivatives of
order $i$ of the function $u$, we denote
\begin{equation*}
\|u\|_{L^2(\Omega)}=\rho_0^{-1}\left(\int_\Omega u^2\right)
^{\frac{1}{2}},
\end{equation*}
\begin{equation*}
\|u\|_{H^k(\Omega)}= \rho_0^{-1} \left ( \sum_{i=0}^{k}
\rho_0^{2i} \int_\Omega |\nabla^i u|^2 \right )^{ \frac{1}{2} }, \quad k=0,1,2,...
\end{equation*}

Moreover, for $k=0,1,2,...$, and $s \in (0,1)$, we denote
\begin{equation*}
\|u\|_{H^{k+s}(\Omega)}= \|u\|_{H^k(\Omega)} + \rho_0^{k+s-1}[\nabla^k u]_{s},
\end{equation*}
where the semi-norm $[\ \!\cdot\ \!]_s$ is given by
\begin{equation}
   \label{eq:2.notation_0}
[\nabla^k u]_{s} = \left ( \int_\Omega \int_\Omega \frac{
|\nabla^k u(x)-\nabla^ku(y)|^2 }{ |x-y|^{2+2s} }\ dx \ dy \right )^{ \frac{1}{2}  }.
\end{equation}

We denote with $H^k_0(\Omega) $ the closure of $C^{\infty}_0(\Omega)$ in  $H^k(\Omega)$ norm and with  $H^{-k}(\Omega)$ the dual space of $H^{k}_0(\Omega)$.

We denote by $\M^{2}, \M^{3}$ the Banach spaces of second order and third order tensors and by 
$\widehat{\M}^{2}, \widehat{\M}^{3}$ the corresponding subspaces of tensors having components invariant with respect to permutations of all the indexes.

Let ${\cal L} (X, Y)$ be the space of bounded linear
operators between Banach spaces $X$ and $Y$. Given $\K\in{\cal L} ({\M}^{2},{\M}^{2})$ and $A,B\in \M^{2}$, we use the following notation 
\begin{equation}
	\label{eq:2.notation_1}
	({\K}A)_{ij} = \sum_{l,m=1}^{2} K_{ijlm}A_{lm},
\end{equation}
\begin{equation}
	\label{eq:2.notation_2}
	A \cdot B = \sum_{i,j=1}^{2} A_{ij}B_{ij},
\end{equation}

Similarly, given $\K\in{\cal L} ({\M}^{3},
{\M}^{3})$ and $A,B\in \M^{3}$, we denote
\begin{equation}
	\label{eq:2.notation_1bis}
	({\K}A)_{ijk} = \sum_{l,m,n=1}^{2} K_{ijklmn}A_{lmn},
\end{equation}
\begin{equation}
	\label{eq:2.notation_2bis}
	A \cdot B = \sum_{i,j,k=1}^{2} A_{ijk}B_{ijk},
\end{equation}

Moroever, for any $A\in \M^{n}$, with $ n=2,3$, we shall denote 
\begin{equation}
	\label{eq:2notation_3}
	|A|= (A \cdot A)^{\frac {1} {2}}.
\end{equation}

\section{The direct problem}
\label{Direct}

Let us consider a nanoplate $\Omega \times  \left ( -\frac{t}{2}, \frac{t}{2}   \right )$ with middle surface $\Omega$ represented by a bounded domain of $\R^2$ and having constant thickness $t$, $t << diam (\Omega)$. We assume that the boundary $\partial \Omega$ of $\Omega$ is of class $C^{2,1}$ with constants $\rho_0$, $M_0$ and that 
\begin{equation}
	\label{eq:1-1}
	|\Omega| \leq M_1 \rho_0^2,
\end{equation}
where $M_1$ is a positive constant.

The nanoplate is resting on a Winkler soil with subgrade reaction
coefficient
\begin{equation}
    \label{eq-3-0}
    \kappa \in L^\infty(\Omega),\quad  0\leq \kappa(x) \leq \frac{\overline{\kappa}}{\rho_0} \ \
    \hbox{a.e. in } \Omega,
\end{equation}
with $\overline{\kappa}$ a given positive constant.

The boundary $\partial \Omega$ is clamped and we assume that a
concentrated force of intensity $f$ is acting at a point $P_0 \in \Omega$ along a
direction orthogonal to the middle surface $\Omega$,
\begin{equation}
    \label{eq-f=}
    f=\rho_0^2\overline{f},
\end{equation}
with $\overline{f}$ a positive constant and
\begin{equation}
    \label{eq-6-1}
    dist(P_0, \partial \Omega) \geq d\rho_0,
\end{equation}
for some positive constant $d$.

 According to the Kirchhoff-Love's  theory of nanoplates subject to infinitesimal
deformation, the statical equilibrium of the plate is described by
the following Dirichlet boundary value problem  \cite{K-M-Z-2022}

\begin{center}
	{\small
		\( {\displaystyle 
			\begin{array}{lr}
				\divrg (\divrg ( (\mathbb P +\mathbb P ^h)\nabla^2 w)) 
				-\divrg (\divrg (\divrg (\mathbb Q \nabla^3 w)))
				+ \kappa w= f\delta_{P_0},& \mathrm{in}\ \Omega,
				\vspace{0.25em}\\
				w =0, & \mathrm{on}\ \partial \Omega,
				\vspace{0.25em}\\
				w,_n =0, & \mathrm{on}\ \partial \Omega,
				\vspace{0.25em}\\
				w,_{nn} =0, & \mathrm{on}\ \partial \Omega.
				\vspace{0.25em}\\
			\end{array}
		}
		\) \vskip -10.7em
		\begin{eqnarray}
			& & \label{eq--3I-1}\\
			& & \label{eq---3I-2}\\
			& & \label{eq--3I-3}\\
			& & \label{eq--3I-4}
		\end{eqnarray}
	} 
\end{center}

where for the sake of simplicity we adopted the notation 
\begin{eqnarray}
\divrg (\divrg ( (\mathbb P +\mathbb P ^h)\nabla^2 w)) 
				-\divrg (\divrg (\divrg (\mathbb Q \nabla^3 w)))
				\end{eqnarray}	
to denote the following sixth order elliptic operator	
\begin{eqnarray}		
				\displaystyle{\frac{\partial^2 }{\partial x_i \partial x_j}\left (
	(P_{ijlm}+P_{ijlm}^h   ) \frac{\partial^2 w}{\partial x_l \partial x_m}
	-
	\frac{\partial }{\partial x_k} \left ( Q_{ijklmn}\frac{\partial^3 w}{\partial x_l \partial x_m \partial x_n} \right )
	\right )	}			
\end{eqnarray}				
where the summation over repeated indexes $i,j,k,l,m,n=1,2$ is assumed.

\
\noindent
On the elasticity tensors  $\mathbb{P}$, $\mathbb{P}^h$, $\mathbb{Q}$ we make the following assumptions: 

\noindent
 
\textit{i) Boundedness}
\begin{equation}
    \label{eq-boundedness}
    \|\mathbb P\|_{L^{\infty}( \Omega, \mathcal{L}(\widehat{\M}^2, \widehat{\M}^2   )  )}, \ \|\mathbb P^h \|_{L^{\infty}( \Omega, \mathcal{L}(\widehat{\M}^2, \widehat{\M}^2   )  )} \leq A_1\rho_0^3,
		\quad \|\mathbb Q \|_{L^{\infty}( \Omega, \mathcal{L}(\widehat{\M}^3, \widehat{\M}^3   )  )
		} \leq A_1\rho_0^5,
\end{equation}
where $A_1$ is a positive constant.

\textit{ii) Isotropy}

\begin{equation}
	\label{eq:6-1}
	P_{\alpha \beta \gamma \delta}= B((1-\nu)\delta_{\alpha \gamma} \delta_{\beta \delta} + \nu \delta_{\alpha \beta} \delta_{\gamma\delta}
	),
\end{equation}
\begin{equation}
	\label{eq:6-2}
	P_{\alpha \beta \gamma \delta}^h= (2a_2+5a_1)\delta_{\alpha \gamma} \delta_{\beta \delta} + (-a_1 -a_2 +a_0) \delta_{\alpha \beta} \delta_{\gamma\delta},
\end{equation}
\begin{multline}
	\label{eq:7-0}
	Q_{ijklmn}= \frac{1}{3}(b_0 -3b_1)\delta_{ij}\delta_{kn}\delta_{lm}+
	\\
	+ \frac{1}{6}(b_0 -3b_1)
	(
	\delta_{ik} ( \delta_{jl} \delta_{mn} + \delta_{jm}\delta_{ln}   )
	+
	\delta_{jk} ( \delta_{il} \delta_{mn} + \delta_{im}\delta_{ln}   )
	)
	+
	Q_8 
	(
	\delta_{kn} ( \delta_{il}\delta_{jm} +\delta_{im}\delta_{jl}    )
	)
	+
	\\
	+
	Q_9
	(
	\delta_{jn} ( \delta_{il} \delta_{km} + \delta_{im}\delta_{kl}   )
	+
	\delta_{in} ( \delta_{jl} \delta_{km} + \delta_{jm}\delta_{kl}   )
	),	
\end{multline}
where $2(Q_8+2Q_9)=5b_1$.

The bending stiffness (per unit length) $B=B(x)$ is given by the function
\begin{equation}
	\label{eq:7-1}
	B(x) = \frac{t^3 E(x)}{12(1-\nu^2(x))}, \quad \hbox{a.e. in } \Omega,
\end{equation}
where the Young's modulus $E$ and the Poisson's coefficient $\nu$ of the material can be written in terms of the Lamé moduli $\mu$ and $\lambda$ as follows
\begin{equation}
	\label{eq:7-2}
	E(x) = \frac{\mu(x)(2\mu(x)+3\lambda(x))}{ \mu(x) +\lambda(x)  },
	\quad 
	\nu(x) = \frac{\lambda(x)}{2(\mu(x)+\lambda(x))}.
\end{equation}
The coefficients $a_i(x)$, $i=0,1,2$, are given by 
\begin{equation}
	\label{eq:7-3}
	a_0(x)=2\mu(x)t\mathit{l}_0^2, \quad a_1(x) = \frac{2}{15}\mu(x)t  \mathit{l}_1^2, \quad a_2(x) = \mu(x)t\mathit{l}_2^2 \quad \hbox{a.e. in } \ \Omega,
\end{equation}
where the material length scale parameters $\mathit{l}_i$ are assumed to be positive constants. We denote

\begin{equation}
	\label{eq:l}
	l=\min\{l_0,l_1, l_2\}.
\end{equation}

The coefficients $b_i(x)$, $i=0,1$, are given by 
\begin{equation}
	\label{eq:7-4}
	b_0(x)=2\mu(x)\frac{t^3}{12}\mathit{l}_0^2, \quad 
	b_1(x)=\frac{2}{5}\mu(x)\frac{t^3}{12}\mathit{l}_1^2
	\quad \hbox{a.e. in } \ \Omega.
\end{equation}

\noindent
\textit{iii) Strong convexity for $\mathbb{P}+\mathbb{P}^h$, $\mathbb{Q}$.}

We assume the following ellipticity conditions on $\mu$ and $\lambda$:
\begin{equation}
	\label{eq:8-1}
	\mu(x) \geq \alpha_0 >0, \quad 2\mu(x)+3\lambda(x) \geq \gamma_0 >0 \quad \hbox{a.e. in } \ \Omega,
\end{equation}
where $\alpha_0$, $\gamma_0$ are positive constants. By \eqref{eq:7-3}, \eqref{eq:7-4} and \eqref{eq:8-1} we also have
\begin{equation}
	\label{eq:8-2}
	a_i(x) \geq t \mathit{l}^2 \alpha_0^h >0, \ i=0,1,2, \quad
	b_j(x) \geq t^3 \mathit{l}^2 \beta_0^h >0, \ j=0,1, \quad \hbox{a.e. in } \ \Omega,
\end{equation}
where $\alpha_0^h = \frac{2}{15}\alpha_0$ and $\beta_0^h = \frac{1}{30} \alpha_0$.

By \eqref{eq:8-1}, \eqref{eq:8-2} we obtain the following strong convexity conditions on $\mathbb{P}+\mathbb{P}^h$ and $\mathbb{Q}$. For every $A \in \widehat{\M}^2$ we have 
\begin{equation}
	\label{eq:8-3}
	(\mathbb{P}+\mathbb{P}^h) A \cdot A \geq t(t^2+l^2) 
	\xi_{\mathbb{P}} |A|^2 \quad \hbox{a.e. in } \ \Omega;
\end{equation}
for every $B \in \widehat{\M}^3$ we have 
\begin{equation}
	\label{eq:8-4}
	\mathbb{Q} B \cdot B \geq t^3l^2 
	\xi_{\mathbb{Q}} |B|^2 \quad \hbox{a.e. in } \ \Omega;
\end{equation}
where $\xi_{\mathbb{P}}$, $\xi_{\mathbb{Q}}$ are positive constants only depending on $\alpha_0$ and $\gamma_0$.

\begin{prop}
\label{prop-4-1}
Under the above assumptions, there exists a unique weak solution
$w \in H^3_0(\Omega)$ to \eqref{eq--3I-1}--\eqref{eq--3I-4} and it satisfies
%
\begin{equation}
    \label{eq-4-1}
    \|w\|_{H^3(\Omega)} \leq C\overline{f}\rho_0,
\end{equation}
where the constant $C>0$ only depends on $\frac{t}{\rho_0}$, $\frac{l}{\rho_0}$, $\alpha_0$, $\gamma_0$, $M_0$.
\end{prop}
\begin{proof}
A weak solution of the problem \eqref{eq--3I-1}--\eqref{eq--3I-4}
is a function $w \in H_0^3(\Omega)$ satisfying
\begin{equation}
    \label{eq-4-2}
    \int_\Omega (\mathbb P+\mathbb P^h) \nabla^2 w \cdot \nabla^2 v + \int_\Omega \mathbb Q \nabla^3 w \cdot \nabla^3 v
    +\int_\Omega \kappa wv = f v(P_0), \quad \hbox{for
    every } v \in H_0^3 (\Omega).
\end{equation}
Let us notice that
\begin{equation}
    \label{eq-4-3}
    H_0^3(\Omega) \subset C^{0,1}(\overline{\Omega}),
\end{equation}
and, therefore, the linear functional
\begin{equation*}
F:H^3_0(\Omega)\rightarrow\R
\end{equation*}
\begin{equation*}
F(v)= f v(P_0)
\end{equation*}
is well-defined and bounded.
The symmetric bilinear form

\begin{equation*}
B:H^3_0(\Omega)\times H^3_0(\Omega)\rightarrow\R
\end{equation*}
\begin{equation}
	\label{eq-B(w,v)}
B(w,v)= \int_\Omega (\mathbb P +\mathbb P^h)\nabla^2 w \cdot \nabla^2 v +
\int_\Omega \mathbb Q\nabla^3 w \cdot \nabla^3 v
+\int_\Omega \kappa wv  
\end{equation}
is bounded and coercive. Precisely,
\begin{equation}
	\label{eq-B_bounded}
B(w,v)\leq C\rho_0 \|w\|_{H^3(\Omega)}\|v\|_{H^3(\Omega)},
\end{equation}
with $C$ only depending on $\overline{\kappa}$ and $A_1$;
\begin{equation}
	\label{eq-B_coercive}
B(w,w)\geq \int_\Omega \mathbb Q \nabla^3 w\cdot \nabla^3 w \geq C\rho_0 \|w\|_{H^3(\Omega)}^2,
\end{equation}
where we have used  the strong convexity conditions \textit{iii)} and the Poincar\'{e} inequality, and $C$ only depends on $\frac{t}{\rho_0}$, $\frac{l}{\rho_0}$, $\alpha_0$, $\gamma_0$, $M_0$.

The variational formulation \eqref{eq-4-2} becomes 
\begin{equation*}
B(w,v)=F(v), \quad \hbox{for every } v\in H^3_0(\Omega),
\end{equation*}
and by Riesz representation theorem a weak solution $w\in H^3_0(\Omega)$ exists and is unique.

By
\eqref{eq-B_coercive}, \eqref{eq-4-3} and recalling \eqref{eq-f=}, the estimate \eqref{eq-4-1}
follows.
\end{proof}

\begin{lem}
\label{lem-8a-1}
Under the above hypotheses, let $w \in
H_0^3(\Omega)$ be the solution to \eqref{eq--3I-1}--\eqref{eq--3I-4}.
There exists $\overline{\sigma} >0$, only depending on $\frac{t}{\rho_0}$, $\frac{l}{\rho_0}$, $\alpha_0$, $\gamma_0$, 
$M_0$, $M_1$, $A_1$, $\overline{\kappa}$, $d$, such that
\begin{equation}
    \label{eq-8a-0}
    w(x)\geq Cd^4{\bar{f}}{\rho_0},\quad \forall x\in B_{2\overline{\sigma}\rho_0}(P_0),
\end{equation}
where $C>0$ only depends on $\frac{t}{\rho_0}$, $\frac{l}{\rho_0}$, $\alpha_0$, $\gamma_0$, $M_0$, $M_1$, $A_1$, $\overline{\kappa}$;

\begin{equation}
    \label{eq-8a-1}
    \int_{B_{2\sigma \rho_0}(P_0)\setminus B_{\sigma
    \rho_0}(P_0)} w^2 \geq C \sigma^2 d^4 \rho_0^2
    \|w\|_{H^3(\Omega)}^2,\quad \hbox{for every }\sigma,\ 0 <
\sigma \leq \overline{\sigma}
\end{equation}
where $C>0$ only depends on $\frac{t}{\rho_0}$, $\frac{l}{\rho_0}$, $\alpha_0$, $\gamma_0$, $M_0$, $M_1$, $\overline{\kappa}$, $A_1$.
\end{lem}
\begin{proof}
By \eqref {eq-4-2} and \eqref{eq-B_bounded}, we have that for every $v
\in H_0^3(\Omega)$
\begin{equation}
    \label{eq-8a-2}
    |v(P_0)| \leq C \frac{\rho_0}{f}\|v\|_{H^3(\Omega)}\|w\|_{H^3(\Omega)},
\end{equation}
with $C$ only depending on $\overline{\kappa}$ and $A_1$, so that
\begin{equation}
    \label{eq-8b-1}
    \|\delta_{P_0}\|_{H^{-3}(\Omega)} = \rho_0^{-2}\sup_
    {\overset{\scriptstyle v\in H_0^3(\Omega)}{\scriptstyle
     v\neq 0}}
     \frac{|v(P_0)|}{ \|v\|_{H^3(\Omega)}} \leq \frac{C}{\rho_0 f}
    \|w\|_{H^3(\Omega)},
\end{equation}
where $C>0$ only depends on $\overline{\kappa}$ and $A_1$.

Next, let us estimate $\|\delta_{P_0}\|_{H^{-3}(\Omega)}$ {}from below.
By \eqref{eq:1-1} and \eqref{eq-6-1}, it follows that
\begin{equation}
    \label{eq-d less}
    d\leq \sqrt{\frac{M_1}{\pi}}.
\end{equation}
Let us set
\begin{equation}
    \label{eq-D=}
    D=\max\left\{\sqrt{\frac{M_1}{\pi}}, 1\right\},
\end{equation}
\begin{equation}
    \label{eq-bar d=}
    \bar d=\frac{d}{D}\leq 1.
\end{equation}
Since $\Omega\supset B_{d\rho_0}(P_0)$, $H^3_0(B_{d\rho_0}(P_0))\subset H^3_0(\Omega)$ and we can estimate
\begin{multline}
    \label{eq-stima dal basso delta}
    \|\delta_{P_0}\|_{H^{-3}(\Omega)}=\rho_0^{-2}\sup_
    {\overset{\scriptstyle v\in H_0^3(\Omega)}{\scriptstyle
     v\neq 0}}
     \frac{|v(P_0)|}{ \left(\int_\Omega \rho_0^{-2}v^2+|\nabla v|^2+\rho_0^2|\nabla^2 v|^2
		+\rho_0^4|\nabla^3 v|^2\right)^{1/2}}\geq\\
		\geq
		\rho_0^{-2}\sup_
    {\overset{\scriptstyle v\in H_0^3(B_{d\rho_0}(P_0))}{\scriptstyle
     v\neq 0}}
     \frac{|v(P_0)|}{ \left(\int_{B_{d\rho_0}(P_0)} \rho_0^{-2}v^2+|\nabla v|^2+\rho_0^2|\nabla^2 v|^2
		+\rho_0^4|\nabla^3 v|^2\right)^{1/2}}.
\end{multline}
Let us estimate {}from above the square of last denominator. Since $\bar d\leq 1$, we have
\begin{multline}
    \label{eq-stima dall'alto denominatore}
   \int_{B_{d\rho_0}(P_0)} \rho_0^{-2}v^2+|\nabla v|^2+\rho_0^2|\nabla^2 v|^2
		+\rho_0^4|\nabla^3 v|^2\leq\\
		\leq \int_{B_{d\rho_0}(P_0)} \bar d^{-6}\rho_0^{-2}v^2+\bar d^{-4}|\nabla v|^2+\bar d^{-2}\rho_0^2|\nabla^2 v|^2
		+\rho_0^4|\nabla^3 v|^2=\\
		=\bar d^{-4}\left(\int_{B_{d\rho_0}(P_0)} \bar d^{-2}\rho_0^{-2}v^2+|\nabla v|^2+\bar d^{2}\rho_0^2|\nabla^2 v|^2
		+\bar d^{4}\rho_0^4|\nabla^3 v|^2\right)=\\
		=\frac{D^4}{d^4}\left(\int_{B_{d\rho_0}(P_0)} \frac{D^2}{(d\rho_0)^2}v^2+|\nabla v|^2+
		\frac{(d\rho_0)^2}{D^2}|\nabla^2 v|^2
		+\frac{(d\rho_0)^4}{D^4}|\nabla^3 v|^2\right)\leq\\
	  \leq 	
		\frac{D^6}{d^4}\left(\int_{B_{d\rho_0}(P_0)} \frac{1}{(d\rho_0)^2}v^2+|\nabla v|^2+
		(d\rho_0)^2|\nabla^2 v|^2
		+(d\rho_0)^4|\nabla^3 v|^2\right).
\end{multline}
where in the last step we have used $D\geq 1$.

Given $v\in H_0^3(B_{d\rho_0}(P_0))$, let us consider the function $\widetilde{v}\in H_0^3(B_1)$ defined by $\widetilde{v}(x)=v(P_0+d\rho_0 x)$.
It is straightforward that 
\begin{equation}
   \int_{B_1(0)}\widetilde{v}^2+|\nabla \widetilde{v}|^2+|\nabla^2 \widetilde{v}|^2+
	|\nabla^3 \widetilde{v}|^2=\int_{B_{d\rho_0}(P_0)}
	\dfrac{v^2}{(d\rho_0)^2}
	+|\nabla v|^2+
	(d\rho_0)^{2}|\nabla^2 v|^2+
	(d\rho_0)^{4}|\nabla^3 v|^2.
\end{equation}

Therefore, introducing the absolute constant
\begin{equation}
    \label{eq-H}
    H=\sup_
    {\overset{\scriptstyle \widetilde{v}\in H_0^3(B_1)}{\scriptstyle
     \widetilde{v}\neq 0}}
     \frac{|\widetilde{v}(0)|}{ \left(\int_{B_1} \widetilde{v}^2+|\nabla \widetilde{v}|^2+|\nabla^2 \widetilde{v}|^2
		+|\nabla^3 \widetilde{v}|^2\right)^{1/2}},
\end{equation}
by \eqref{eq-stima dal basso delta}, \eqref{eq-stima dall'alto denominatore} and \eqref{eq-H} we have
\begin{equation}
    \label{eq-stima dal basso deltaBIS}
    \|\delta_{P_0}\|_{H^{-3}(\Omega)}\geq \frac{H}{\rho_0^2}\frac{d^2}{D^3}.
\end{equation}
By comparing \eqref{eq-8b-1} and \eqref{eq-stima dal basso deltaBIS}
\begin{equation}
    \label{eq-8b-3}
    \|w\|_{H^3(\Omega)} \geq \frac{C}{\rho_0} d^2f,
\end{equation}
where $C>0$ only depends on $\overline{\kappa}$, $M_1$, $A_1$.

By \eqref{eq-4-2} and \eqref{eq-B_coercive},
\begin{equation}
    \label{eq-w dal basso in P_0}
    w(P_0)\geq C\frac{\rho_0}{f}\|w\|_{H^3(\Omega)}^2,
\end{equation}
with $C$ only depending on $\frac{t}{\rho_0}$, $\frac{l}{\rho_0}$, $\alpha_0$, $\gamma_0$, $M_0$.
By estimating one of the factors $\|w\|_{H^3(\Omega)}$ in the above inequality with \eqref{eq-8b-3}, we obtain
\begin{equation}
    \label{eq-8b-5}
    w(P_0) \geq C d^2 \|w\|_{H^3(\Omega)},
\end{equation}
where $C$ only depends on $\frac{t}{\rho_0}$, $\frac{l}{\rho_0}$, $\alpha_0$, $\gamma_0$, $M_0$, $M_1$, $A_1$, $\overline{\kappa}$.
By the embedding inequality \eqref{eq-4-3}, we have
\begin{equation}
    \label{eq-8b-6}
    w(P_0) \geq c_0 d^2 \|w\|_{C^{0,1}(\overline{\Omega})},
\end{equation}
where $c_0 >0$ only depend on $\frac{t}{\rho_0}$, $\frac{l}{\rho_0}$, $M_0$, $M_1$, $A_1$,
$\alpha_0$, $\gamma_0$, $\overline{\kappa}$. Let
\begin{equation}
    \label{eq-8c-1}
    \overline{\sigma} = \min \left ( \frac{d}{4}, \frac{c_od^2}{4}\right ).
\end{equation}
Let us notice that, by this choice of $\overline{\sigma}$,
$dist(P_0, \partial \Omega) \geq 4 \overline{\sigma} \rho_0$ and,
recalling \eqref{eq-8b-6}, we have that, for every $x \in
B_{2\overline{\sigma}\rho_0}(P_0)$,
\begin{eqnarray}
    \label{eq-8c-2}
    w(x) \geq w(P_0) - |w(x)-w(P_0)| \geq w(P_0) -
    2\overline{\sigma}\|w\|_{C^{0,1}(\overline{\Omega})} \geq
    \frac{w(P_0)}{2}.
\end{eqnarray}

Now, \eqref{eq-8a-0} follows {}from \eqref{eq-8b-3} , \eqref{eq-w dal basso in P_0} and \eqref{eq-8c-2}.
The integral estimate \eqref{eq-8a-1} follows {}from \eqref{eq-8b-5} and \eqref{eq-8c-2}.

\end{proof}

As is customary, to deal with the inverse problem we need a preliminary analysis of the fine properties of the solutions of the underlying equations. For this purpose we state a regularity result for the solution to a slight variation of the equation at hand with coefficients belonging to fractional Sobolev spaces.

\begin{prop}[$H^s$-regularity]
\label{prop-5a-1}
Let $\Omega$ be a bounded domain in $\R^2$ with boundary of Lipschitz
class with constants $\rho_0$, $M_0$, satisfying \eqref{eq:1-1}. Given $g \in H^s(\Omega)$, for some $s\in (0,1)$, let $w \in H^3(\Omega)$ be a weak solution
to
\begin{eqnarray}
    \label{eq-5a-1}
	\divrg (\divrg ( (\mathbb P +\mathbb P ^h)\nabla^2 w)) 
				-\divrg (\divrg (\divrg (\mathbb Q \nabla^3 w)))=g, \quad \mbox{in } \Omega, 
				\end{eqnarray}

where $\mathbb P, \mathbb P^h, \Q$ satisfy the isotropy condition  \textit{ii)} , the strong convexity conditions \textit{iii)} and for some $s\in(0,1)$ 
\begin{equation}
    \label{eq-5a-2}
    \|\mathbb P\|_{W^{2,\infty}(\Omega)}, \ \|\mathbb P^h \|_{W^{2,\infty}(\Omega)} \leq A_2\rho_0^3,
		\quad \|\mathbb Q \|_{W^{3,\infty}(\Omega)} \leq A_2\rho_0^5,
\end{equation}
\begin{equation}
    \label{eq-5a-3}
    \|\mathbb P\|_{H^{2+s}(\Omega)}, \ \|\mathbb P^h \|_{H^{2+s}(\Omega)} \leq A_3\rho_0^3, \quad \|\mathbb Q  \|_{H^{3+s}(\Omega)} \leq A_3\rho_0^5 \ .
\end{equation}

%
Then, for every $\sigma > 0$, we have
\begin{equation}
    \label{eq-5a-4}
    \|w \|_{H^{6+s}(\Omega_{\sigma \rho_0})} \leq C\left (
    \|w\|_{H^3(\Omega_{\frac{\sigma}{2}\rho_0})} + \rho_0 \|g\|_{H^{s}(\Omega)}
    \right ),
\end{equation}
where the constant $C>0$ only depends on $\frac{t}{\rho_0}$, $\frac{l}{\rho_0}$, $\alpha_0$, $\gamma_0$, $M_0$, $M_1$, $A_2$,
$A_3$, $s$, $\sigma$.
\end{prop}
\begin{proof}
The method of proof is based on the techniques introduced in \cite{Ag,Fo}. 

\end{proof}

\section{The inverse problem}
\label{Inverse}

In order to derive our stability result for the inverse problem we
need further a priori information on the subgrade reaction
coefficient $\kappa$, namely we shall assume 


%
%

%
\begin{equation}
    \label{eq-7-3}
    \|\kappa\|_{L^\infty(\Omega)}+\rho_0^{s-1}[\kappa]_{H^s(\Omega)} \leq
    \frac{\overline{\kappa}}{\rho_0}, \quad \kappa \geq 0 \hbox{ a.e. in } \Omega.
\end{equation}
Hereinafter, we shall refer to $\frac{t}{\rho_0}$, $\frac{l}{\rho_0}$, $\alpha_0$, $\gamma_0$, $d$, $M_0$, $M_1$, $A_2$,
$A_3$, $\overline{\kappa}$, $\overline{f}$, $s$ as the \textit{a priori data}.

\begin{theo}
\label{theo-7-1}

Let $\Omega$ be a bounded domain in $\R^2$ with boundary of
Lipschitz class with constants $\rho_0$, $M_0$, satisfying
\eqref{eq:1-1}. Let $\mathbb P$, $\mathbb P^h$, $\mathbb Q$ satisfying the isotropy conditions \textit{ii)}, the strong convexity conditions \textit{iii)} and the regularity conditions  
\eqref{eq-5a-2}, \eqref{eq-5a-3} for some $s\in (0,1)$. Let $P_0 \in \Omega$
satisfying \eqref{eq-6-1}.
Let $f \in \R$ be given by \eqref{eq-f=} and let $w_i \in H_0^3(\Omega)$, $i=1,2$, be the solution
to

\begin{center}
{\small
\( {\displaystyle 
\begin{array}{lr}
        \divrg (\divrg ( (\mathbb P +\mathbb P ^h)\nabla^2 w_i)) 
				-\divrg (\divrg (\divrg (\mathbb Q \nabla^3 w_i)))
				+ \kappa_iw_i= f\delta_{P_0},& \mathrm{in}\ \Omega,


    \vspace{0.25em}\\
    w_i =0, & \mathrm{on}\ \partial \Omega,
        \vspace{0.25em}\\
    w_i,_n =0, & \mathrm{on}\ \partial \Omega,
          \vspace{0.25em}\\
		w_i,_{nn} =0, & \mathrm{on}\ \partial \Omega,
          \vspace{0.25em}\\
\end{array}
}
\) \vskip -10.7em
\begin{eqnarray}
& & \label{eq-3I-1i}\\
& & \label{eq-3I-2i}\\
& & \label{eq-3I-3i}\\
& & \label{eq-3I-4i}
\end{eqnarray}
}
\end{center}
for $\kappa_i \in L^\infty(\Omega) \cap
H^s(\Omega)$ satisfying \eqref{eq-7-3}.

If, for some $\epsilon>0$,
\begin{equation}
    \label{eq-8-1}
    \|w_1 - w_2\|_{L^2(\Omega)}\leq \epsilon\rho_0 \overline{f},
\end{equation}
then for every $\sigma > 0$ we have
\begin{equation}
    \label{eq-8-2}
    \|\kappa_1 - \kappa_2\|_{L^2(\Omega_{\sigma \rho_0})}\leq \frac{C}{\rho_0} \epsilon^\beta,
\end{equation}
where the constants $C>0$ and $\beta \in (0,1)$ only depend on the
a priori data and on $\sigma$.
\end{theo}

Let us now premise to the proof of Theorem \ref{theo-7-1}. 
some auxiliary propositions concerning quantitative versions of the unique continuation principle (Propositions \ref{prop-8-1}
and  Proposition \ref{prop-9-1} below).

\begin{prop}[Lipschitz propagation of smallness]
\label{prop-8-1}
Let $U$ be a bounded domain in $\R^2$ of Lipschitz class with constants
$\rho_0$, $M_0$ and satisfying $|U| \leq M_1 \rho_0^2$. Let $w \in
H^3(U)$ be a solution to
\begin{equation}
    \label{eq-8bis-1}
    \divrg (\divrg ( (\mathbb P +\mathbb P ^h)\nabla^2 w)) 
				-\divrg (\divrg (\divrg (\mathbb Q \nabla^3 w)))
				+ \kappa w = 0, \quad \hbox{in } U,
\end{equation}
where $\mathbb P, \mathbb P^h, \mathbb Q$ satisfy the isotropy conditions \textit{ii)}, the strong convexity conditions \textit{iii)} and the regularity conditions \eqref{eq-5a-2} in $U$, and $\kappa$  
satisfies \eqref{eq-3-0} in $U$. Assume
\begin{equation}
   \label{eq-frequency}
\frac{\|w\|_{H^{ \frac{1}{2} } (U)}}{\|w\|_{L^2(U)}}\leq N.
\end{equation}
There exists a constant $c_1 >
1$, only depending on $\frac{t}{\rho_0}$,$\frac{l}{\rho_0}$, $\alpha_0$, $\gamma_0$, $A_2$,  and $\overline{\kappa}$, such
that, for every $\tau >0$ and for every $x \in U_{c_1 \tau
\rho_0}$, we have
\begin{equation}
    \label{eq-9-1}
    \int_{B_{\tau \rho_0}(x)} w^2 \geq c_\tau \int_U w^2,
\end{equation}
where $c_\tau >0$ only depends on  $\frac{t}{\rho_0}$,$\frac{l}{\rho_0}$, $\alpha_0$, $\gamma_0$, $M_0$, $M_1$, $A_2$, $\overline{\kappa}$, $\tau$ and on $N$.
\end{prop}
The proof of the above proposition is based on the three spheres
inequality obtained in  \cite{M-R-S-V-2024}  and the argument in  \cite[Proposition 6.1]{MRSV2023}. 
\begin{prop}[$A_p$ property]
\label{prop-9-1}
In the same hypotheses of Proposition \ref{prop-8-1}, there exists
a constant $c_2 > 1$, only depending on $\frac{t}{\rho_0}$, $\frac{l}{\rho_0}$, $\alpha_0$, $\gamma_0$,
$M_0$,
$M_1$, $A_2$, $\overline{\kappa}$, such that, for every $\tau
> 0$ and for every $x \in U_{c_2\tau \rho_0}$, we have
\begin{equation}
    \label{eq-9-2}
    \left (
    \frac{1}{|B_{\tau \rho_0}(x)|}
    \int_{B_{\tau \rho_0}(x)} |w|^2
    \right )
    \left (
    \frac{1}{|B_{\tau \rho_0}(x)|}
    \int_{B_{\tau \rho_0}(x)} |w|^{ - \frac{2}{p-1}  }
    \right )^{  p-1  } \leq B,
\end{equation}
where $B>0$ and $p > 1$ only depend on $\frac{t}{\rho_0}$,$\frac{l}{\rho_0}$, $\alpha_0$, $\gamma_0$, 
$M_0$, $M_1$, $A_2$,
$\overline{\kappa}$, $\tau$ and on the quantity $N$ defined in \eqref{eq-frequency}.
\end{prop}
The proof of the above proposition follows from the doubling inequality
obtained in  \cite{M-R-S-V-2024} (see also \cite{MRSV2023}), by applying the arguments in \cite{G-L}.

\begin{proof}[Proof of Theorem \ref{theo-7-1}]
If $\epsilon\geq 1$, then the proof of \eqref{eq-8-2} is trivial in view of
\eqref{eq-3-0}. Therefore it is not restrictive to assume $0<\epsilon<1$.

The difference
\begin{equation}
    \label{eq-10-1}
    w=w_1 - w_2
\end{equation}
of the solutions to \eqref{eq-3I-1i}--\eqref{eq-3I-4i} for $i=1,2$
is a weak solution to  the boundary value problem

\begin{center}
	{\small
		\( {\displaystyle 
			\begin{array}{lr}
				\divrg (\divrg ( (\mathbb P +\mathbb P ^h)\nabla^2 w)) 
				-\divrg (\divrg (\divrg (\mathbb Q \nabla^3 w)))
				+ \kappa_2w= (\kappa_2-\kappa_1)w_1,& \mathrm{in}\ \Omega,
				\vspace{0.25em}\\
				w =0, & \mathrm{on}\ \partial \Omega,
				\vspace{0.25em}\\
				w,_n =0, & \mathrm{on}\ \partial \Omega,
				\vspace{0.25em}\\
				w,_{nn} =0, & \mathrm{on}\ \partial \Omega.
				\vspace{0.25em}\\
			\end{array}
		}
		\) \vskip -10.7em
		\begin{eqnarray}
		& & \label{eq-3I-1iBIS}\\
& & \label{eq-3I-2iBIS}\\
& & \label{eq-3I-3iBIS}\\
& & \label{eq-3I-4iBIS} 
		\end{eqnarray}
	} 
\end{center}



Obviously, it is not restrictive to assume that
$\sigma \leq \overline{\sigma}$, with
$\overline{\sigma}$ defined in \eqref{eq-8c-1}.

In view of the regularity result in Proposition \ref{prop-5a-1} we notice that the solution to \eqref{eq-3I-1iBIS}-\eqref{eq-3I-4iBIS} belongs to $H^{6+s}(\Omega)$.

 By \eqref{eq-3I-1iBIS}, we have
\begin{equation}
    \label{eq-10-5}
    \int_{\Omega_{\sigma \rho_0}} (\kappa_2 -\kappa_1)^2 w_1^2 \leq 2(I_1
    +I_2),
\end{equation}
where
\begin{equation}
    \label{eq-10-6}
    I_1 = \int_{\Omega_{\sigma \rho_0}} \kappa_2^2 w^2 ,
\end{equation}
\begin{equation}
    \label{eq-10-7}
    I_2 = \int_{\Omega_{\sigma \rho_0}} \left[ \left| \divrg (\divrg ( (\mathbb P+\mathbb P^h) \nabla^2 w))\right| + \left|\divrg (\divrg(\divrg(\mathbb Q\nabla^3 w)))\right| \right
    ]^2.
\end{equation}
By \eqref{eq-3-0} and \eqref{eq-8-1}, we have
\begin{equation}
    \label{eq-10-8}
    I_1 \leq \overline{\kappa}^2\overline{f}^2\rho_0^2 \epsilon^2.
\end{equation}
By
\eqref{eq-5a-2}, we have
\begin{equation}
    \label{eq-11-1}
    I_2 \leq C A_2^2\|w\|_{H^6(\Omega_{\sigma \rho_0})}^2,
\end{equation}
with $C>0$ an absolute constant.
Let $g= (\kappa_2-\kappa_1)w_1 -\kappa_2w$. Note that, by \eqref{eq-4-1}, \eqref{eq-4-3}
and \eqref{eq-7-3},

\begin{equation}
    \label{eq-11-2}
    \|g\|_{H^s(\Omega)} \leq C \overline{\kappa}\,\overline{f},
\end{equation}
where $C>0$ only depends on $\frac{t}{\rho_0}$, $\frac{l}{\rho_0}$, $M_0$, $\alpha_0$, $\gamma_0$ and $s$.
By applying Proposition 
\ref{prop-5a-1} and estimate \eqref{eq-4-1} we
have
\begin{equation}
    \label{eq-11-3}
    \|w\|_{H^{6+s}(\Omega_{\sigma \rho_0})} \leq C \overline{f}
    \rho_0,
\end{equation}
with $C>0$ only depending on $\frac{t}{\rho_0}$, $\frac{l}{\rho_0}$, $M_0$, $M_1$, $A_2$, $A_3$,
$\alpha_0$, $\gamma_0$, $\overline{\kappa}$, $s$, $\sigma$. {}From the well-known interpolation
inequality
\begin{equation}
    \label{eq-11-4}
    \|w\|_{H^{6}(\Omega_{\sigma \rho_0})} \leq C \|w\|_{H^{6+s}(\Omega_{\sigma
    \rho_0})}^{  \frac{6}{6+s}  }
    \|w\|_{L^{2}(\Omega_{\sigma
    \rho_0})}^{  \frac{s}{6+s}  },
\end{equation}
and by \eqref{eq-8-1} and \eqref{eq-11-3}, we obtain
\begin{equation}
    \label{eq-11-5}
    \|w\|_{H^{6}(\Omega_{\sigma \rho_0})} \leq C \overline{f}\rho_0
    \epsilon^{  \frac{s}{6+s}  },
\end{equation}
where $C>0$ only depends on $\frac{t}{\rho_0}$, $\frac{l}{\rho_0}$, $M_0$, $M_1$, $A_2$, $A_3$,
$\alpha_0$, $\gamma_0$, $\overline{\kappa}$, $s$, $\sigma$.
{}From
\eqref{eq-10-5}, \eqref{eq-10-8}, \eqref{eq-11-1},
\eqref{eq-11-5}, it follows that
\begin{equation}
    \label{eq-12-1}
    \int_{\Omega_{\sigma \rho_0}} (\kappa_2 -\kappa_1)^2 w_1^2 \leq
    C \overline{f}^2\rho_0^2\epsilon^{ \frac{2s}{6+s}  },
\end{equation}
where $C>0$ only depends on $\frac{t}{\rho_0}$, $\frac{l}{\rho_0}$, $M_0$, $M_1$, $A_2$, $A_3$,
$\alpha_0$, $\gamma_0$, $\overline{\kappa}$, $s$, $\sigma$.

Let us first estimate $|\kappa_2-\kappa_1|$ in a disc centered at $P_0$.
Notice that, by the choice of $\overline{\sigma}$, $\Omega_{\sigma\rho_0}\supset B_{2\overline{\sigma}\rho_0}(P_0)$, for every $\sigma\leq\overline{\sigma}$.
By applying \eqref{eq-8a-0} for $w=w_1$, and by \eqref{eq-12-1} with $\sigma=\overline{\sigma}$, we obtain
\begin{equation}
    \label{eq-12-2}
    \int_{B_{2\overline{\sigma} \rho_0}(P_0)} (\kappa_2 -\kappa_1)^2 \leq
    \frac{C}{d^8}\epsilon^{ \frac{2s}{6+s}  },
\end{equation}
where $C>0$ only depends on $\frac{t}{\rho_0}$, $\frac{l}{\rho_0}$, $M_0$, $M_1$, $A_2$, $A_3$,
$\alpha_0$, $\gamma_0$, $\overline{\kappa}$, $s$, $\sigma$.

Now, let us control $|\kappa_2-\kappa_1|$ in
\begin{equation}
    \label{eq-13-1}
    \widetilde{\Omega}_{\sigma \rho_0} = \Omega_{\sigma \rho_0}
    \setminus B_{2\overline{\sigma} \rho_0}.
\end{equation}

By applying H\"{o}lder inequality and \eqref{eq-12-1}, we can
write, for every $p > 1$,
\begin{multline}
    \label{eq-13-2}
    \int_{\widetilde{\Omega}_{\sigma \rho_0}}(\kappa_2-\kappa_1)^2 =
    \int_{\widetilde{\Omega}_{\sigma \rho_0}}|w_1|^{ \frac{2}{p}
    }(\kappa_2-\kappa_1)^2|w_1|^{ -\frac{2}{p}  } \leq \\
    \leq
    \left (
    \int_{\widetilde{\Omega}_{\sigma \rho_0}}
    (\kappa_2-\kappa_1)^2 w_1^2
    \right )^{  \frac{1}{p} }
    \left (
    \int_{\widetilde{\Omega}_{\sigma \rho_0}}
    (\kappa_2-\kappa_1)^2 |w_1|^{ - \frac{2}{p-1}  }
    \right )^{  \frac{p-1}{p} } \leq \\
    \leq
    C\overline{f}^{ \frac{2}{p}   }\rho_0^{ \frac{2}{p}   }
    \epsilon^{ \frac{2s}{p(6+s)}   }
    \left (
    \int_{\widetilde{\Omega}_{\sigma \rho_0}}
    (\kappa_2-\kappa_1)^2 |w_1|^{ - \frac{2}{p-1}  }
    \right )^{  \frac{p-1}{p} },
\end{multline}
where $C>0$ only depends on $\frac{t}{\rho_0}$, $\frac{l}{\rho_0}$, $M_0$, $M_1$, $A_2$, $A_3$,
$\alpha_0$, $\gamma_0$, $\overline{\kappa}$, $s$, $\sigma$.

Let us cover $\widetilde{\Omega}_{\sigma \rho_0}$ with internally
non overlapping closed squares $Q_l(x_j)$ with center $x_j$ and side $l
= \frac{\sqrt{2}}{2\max\{2,c_1,c_2\}}\sigma \rho_0$, $j=1,...,J$,
where $c_1$ and $c_2$ have been introduced in Proposition
\ref{prop-8-1} and in Proposition \ref{prop-9-1}, respectively. By
the choice of $l$, denoting $r= \frac{\sqrt{2}}{2}l$,
\begin{equation}
    \label{eq-14-1}
    \widetilde{\Omega}_{\sigma \rho_0} \subset \bigcup_{j=1}^J Q_l(x_j)
    \subset \bigcup_{j=1}^J B_r(x_j) \subset
    \Omega_{\frac{\sigma}{2}
    \rho_0}\setminus B_{\overline{\sigma}\rho_0}(P_0),
\end{equation}
so that
\begin{equation}
    \label{eq-14-2}
    \int_{\widetilde{\Omega}_{\sigma \rho_0}}
    (\kappa_2-\kappa_1)^2 |w_1|^{ - \frac{2}{p-1}  }
    \leq
    \frac{4\overline{\kappa}^2}{\rho_0^2} \int_{\widetilde{\Omega}_{\sigma \rho_0}}
    |w_1|^{ - \frac{2}{p-1}  } \leq \frac{4\overline{\kappa}^2}{\rho_0^2} \sum_{j=1}^J\int_{B_r(x_j)}
    |w_1|^{ - \frac{2}{p-1}  }.
\end{equation}
By applying the $A_p$-property \eqref{eq-9-2} and the Lipschitz
propagation of smallness property \eqref{eq-9-1} to $w=w_1$ in $U
= \Omega \setminus B_{\overline{\sigma}\rho_0}(P_0)$,
with $\tau = \frac{r}{\rho_0}= \frac{\sigma}{2\max\{ 2,c_1,c_2
\}}$, and noticing that, for every $j$, $j=1,...,J$, $dist(x_j,
\partial U) \geq c_i r$, $i=1,2$, we have
\begin{equation}
    \label{eq-14-3}
    \int_{B_r(x_j)}  |w_1|^{ - \frac{2}{p-1}  }
    \leq
    \frac
    { B^{ \frac{1}{p-1}} |B_r(x_j)|    }
    {
    \left (
    \frac{1}{ |B_r(x_j)|  }
    \int_{B_r(x_j)}  |w_1|^2
    \right )^{ \frac{1}{p-1}   }   }
    \leq
    \frac
    { B^{ \frac{1}{p-1} }|B_r(x_j)|    }
    {
    \left (
    \frac{c_\tau}{ |B_r(x_j)|  }
    \int_{\Omega \setminus B_{\overline{\sigma}\rho_0}(P_0)}  |w_1|^2
    \right )^{ \frac{1}{p-1}   }   },
\end{equation}
where $B>0$, $p>1$ and $c_\tau>0$ only depend on $\frac{t}{\rho_0}$, $\frac{l}{\rho_0}$, $M_0$, $M_1$,
$A_2$, $\alpha_0$, $\gamma_0$, $\overline{\kappa}$, $\sigma$ and on the frequency
ratio  $\mathcal{F}= \frac{ \|w_1\|_{ H^{ \frac{1}{2}   }(  \Omega
\setminus B_{\overline{\sigma}\rho_0}(P_0) )   } }
 {  \|w_1\|_{ L^{2}(  \Omega
\setminus B_{\overline{\sigma}\rho_0}(P_0) )   }  }$.
Notice that, since $\Omega \setminus
B_{\overline{\sigma}\rho_0}(P_0)  \supset
B_{2\overline{\sigma}\rho_0}(P_0)  \setminus
B_{\overline{\sigma}\rho_0}(P_0) $, by applying \eqref{eq-8a-1},
we have
\begin{equation}
    \label{eq-15-1}
    \mathcal{F} \leq \frac{C}{\overline{\sigma} d^2},
\end{equation}
where $C>0$ only depends on $\frac{t}{\rho_0}$, $\frac{l}{\rho_0}$, $M_0$, $M_1$, $\alpha_0$, $\gamma_0$, $A_2$, $\overline{\kappa}$.
By applying \eqref{eq-8a-1} and \eqref{eq-8b-3} to estimate {}from
below the denominator in the right hand side of \eqref{eq-14-3}, by
\eqref{eq:1-1} and \eqref{eq-14-2}, we obtain   
\begin{equation}
    \label{eq-15-2}
    \int_{\widetilde{\Omega}_{\sigma \rho_0}}
    (\kappa_2-\kappa_1)^2 |w_1|^{ - \frac{2}{p-1}  }
    \leq
    \frac{C|\Omega|}{\rho_0^2(d^8f^2\rho_0^{-2})^{\frac{1}{p-1}}} \leq \frac{C}{(d^8\overline{f}^2\rho_0^2)^{\frac{1}{p-1}}},
\end{equation}
where $C>0$ only depends on $\frac{t}{\rho_0}$, $\frac{l}{\rho_0}$, $M_0$, $M_1$, $A_2$, $\alpha_0$, $\gamma_0$, $\overline{\kappa}$ and $\sigma$. By \eqref{eq-13-2} and \eqref{eq-15-2} we have
\begin{equation}
    \label{eq-16-1}
    \int_{\widetilde{\Omega}_{\sigma \rho_0}}
    (\kappa_2-\kappa_1)^2
    \leq
    \frac{C}{d^{\frac{8}{p}}} \epsilon^{ \frac{2s}{p(6+s)}  },
\end{equation}
where $C>0$ only depends on $\frac{t}{\rho_0}$, $\frac{l}{\rho_0}$, $M_0$, $M_1$, $A_2$, $A_3$, $\alpha_0$, $\gamma_0$, $\overline{\kappa}$, $s$ and $\sigma$. Finally, by
\eqref{eq-12-2} and \eqref{eq-16-1}, and recalling that $p>1$ and $\epsilon<1$, estimate \eqref{eq-8-2} follows with $\beta = \frac{s}{p(6+s)}$.
\end{proof}

\section{Concluding remarks}
\label{Conclusions}

Modeling the interaction between a nanoplate and the supporting structure is decisive for the definition of a mechanical model capable of predicting the response of the nanoplate accurately. Assuming that the interaction can be described by means of a Winkler elastic model, in this paper we have shown that the nonlinear inverse problem of determining the subgrade coefficient $\kappa$ from the measurement of the transverse deflection induced by a concentrated force admits H\"older-type stability. Key ingredients of the proof are quantitative versions of the unique continuation principle for solutions to the homogeneuous sixth order partial differential equation governing the  statical bending of a nanoplate resting on a Winkler foundation.

We expect that our result can be useful in the development of reconstruction methods for $\kappa$, as it implies the convergence of procedures based on regularization techniques \cite{DeH-Q-S-2012}. We plan to investigate on this issue in a subsequent work.

Winkler's foundation model is based on the assumption that there is no interaction between the elastic springs. This restrictive hypothesis is overcome, for example, in the Winkler-Pasternak model by including the additional term $-\divrg(p \nabla w)$ in the field equation \eqref{eq-intro-3I-1}, where $p \geq 0$ is a second unknown parameter \cite{M-M-B-A-S-2022}. The stable determination of both $\kappa$ and $p$ {}from a pair of measurements of the transverse deflection induced by a single concentrated load assigned at different points is, as far as we know, an open question, and even a general uniqueness result does not appear to be entirely trivial.

\bigskip
\noindent
\textbf{Acknowledgement.} The work of AM was supported by PRIN 2022 n. 2022JMSP2J "Stability, multiaxial fatigue and fatigue life prediction in statics and dynamics of innovative structural and material coupled systems" funded by MUR, Italy, and by the European Union – Next Generation EU.  The work of ES and SV was supported by PRIN 2022 n. 2022B32J5C funded by MUR, Italy, and by the European Union – Next Generation EU. ES has also been supported by Gruppo Nazionale per l'Analisi \text{Matematica,} la Probabilità e le loro applicazioni (GNAMPA) by the grant "Problemi inversi per equazioni alle derivate parziali". 

\bigskip

\end{document}